\documentclass[11pt]{article}
\usepackage{amsmath,amsfonts,amssymb, amsthm,enumerate,graphicx}
\usepackage{tikz,authblk}
\usepackage{subfig}
\usepackage{url}

\newcommand{\lk}[2]{{\rm lk}_{#1}(#2)}
\newcommand{\st}[2]{{\rm st}_{#1}(#2)}

\newtheorem{Lemma}{Lemma}[section]
\newtheorem{Theorem}[Lemma]{Theorem}

\newtheorem{Corollary}[Lemma]{Corollary}
\newtheorem{Remark}[Lemma]{Remark}
\newtheorem{definition}[Lemma]{Definition}

\newtheorem{Question}[Lemma]{Question}
\newtheorem{Conjecture}[Lemma]{Conjecture}

\def\t{\tau}
\def\s{\sigma}
\def\p{\partial}

\def\e{{\varepsilon}}
\def\D{\Delta}
\def\z{\zeta}
\def\lk{\text{\textnormal{lk}}}
\def\st{\text{\textnormal{st}}}
\def\deg{\text{\textnormal{deg}}}
\def\wrt{{\rm with respect to }}

\begin{document}

\title{Embeddings of edge-colored dual graphs of balanced 3- and 4-manifolds}
\author{Biplab Basak$^1$ and Sourav Sarkar}
\date{December 2, 2025}
\maketitle
\vspace{-10mm}
\begin{center}
\noindent{\small Department of Mathematics, Indian Institute of Technology Delhi, New Delhi 110016, India.$^2$}
\end{center}
\footnotetext[1]{Corresponding author}
\footnotetext[2]{{\em E-mail addresses:} \url{biplab@iitd.ac.in} (B.
Basak), \url{sarkarsourav610@gmail.com} (S. Sarkar).}
\hrule

\begin{abstract} 
This article focuses on a class of properly edge-colored graphs, which arise from topological combinatorics, and investigates their embeddings onto surfaces. Specifically, these graphs are known as the dual graphs of balanced normal pseudomanifolds. We introduce the concept of the balanced genus, which represents the smallest genus of a surface onto which the dual graph of a normal pseudomanifold can embed regularly.

As a key result, we establish that for any 3-manifold $ M $ that is not a sphere, the balanced genus satisfies the lower bound $ \mathcal{G}_M \geq m+3 $, where $ m $ is the rank of its fundamental group of $M$. Furthermore, we prove that a 3-manifold $ M $ is homeomorphic to the 3-sphere if and only if its balanced genus $ \mathcal{G}_M $ is at most 3.  Similarly, for 4-manifolds, we establish that if $ M $ is not homeomorphic to a sphere, then its balanced genus is bounded below by $ \mathcal{G}_M \geq 2\chi(M) + 5m + 11 $. Moreover, a 4-manifold $ M $ is PL homeomorphic to the 4-sphere if and only if its balanced genus satisfies $ \mathcal{G}_M \leq 2\chi(M) + 10 $.  

We believe that the balanced genus offers a new perspective in graph theory and combinatorics and will inspire further developments in the field in connection with algebraic combinatorics. To this end, we outline several directions for future research.

\end{abstract}

\noindent {\small {\em MSC 2020\,:} Primary 05C15; Secondary 05E45, 05A20, 05C75.

}

\medskip
\section{Introduction} 
A key feature of balanced triangulations is that they are properly vertex-colorable, ensuring that their dual graphs are edge-colorable. This enables the study of balanced triangulations through their associated dual graphs. Moreover, since these graphs naturally embed on surfaces, we can relate the genus of the embedding surface to the balanced triangulation. This connection provides a flexible approach to studying balanced triangulations in terms of genus, leading to the natural introduction of a new PL invariant.

The {\em dual graph} of a normal $d$-pseudomanifold $\D$ is the graph whose vertices represent the facets of $\Delta$, with edges given by unordered pairs $\{\sigma_1, \sigma_2\}$, where $\sigma_1$ and $\sigma_2$ are two facets of $\D$ that share a common $(d-1)$-simplex $\sigma_1\cap\sigma_2$. We denote the dual graph of a normal $d$-pseudomanifold $\D$ by $\Lambda(\Delta)$. Observe that $\Lambda(\Delta)$ is a simple graph. Moreover, if  $\Delta$ is a balanced normal $d$-pseudomanifold, then the dual graph $\Lambda(\Delta)$ admits a proper edge coloring.

Let $\Delta$ be a balanced normal $d$-pseudomanifold with a proper vertex coloring $\kappa : V(\Delta)\to [d]$, where $[d]=\{0,1,\dots,d\}$. Define an edge coloring $\gamma$ of $\Lambda(\Delta)$ using the color set $[d]$ as follows: for an edge $e = \{\sigma_1, \sigma_2\}$, where the vertices of $\sigma_1\cap\sigma_2$ are colored $0,1,\dots, i-1, i+1, \dots, d$, set $\gamma(e) = i$. Thus, $\gamma$ is a proper edge coloring of $\Lambda(\Delta)$ with the color set  $[d]$. 

Note that, under the aforementioned edge coloring  $\gamma$,  $\Lambda(\Delta)$ forms a $(d+1)$-regular colored graph. An embedding $i : \Lambda(\Delta) \hookrightarrow F$ of $\Lambda(\Delta)$ into a closed surface $F$ is called a {\it regular embedding} if there exists a cyclic permutation $\varepsilon=(\varepsilon_0, \varepsilon_1, \dots, \varepsilon_d)$ of the color set $[d]=\{0, \dots, d\}$ such that the boundary of each face of $i(\Lambda(\Delta))$ forms a bi-colored cycle with colors $\varepsilon_j$ and $ \varepsilon_{j+1}$ for some $j$, where the addition is taken modulo $d+1$.  

A balanced normal pseudomanifold naturally gives rise to a dual graph, which is a regular colored graph. Such graphs admit regular embeddings into surfaces, making it natural to associate balanced normal pseudomanifolds with the genus of the corresponding embedded surface. Motivated by this observation, we introduce the concept of balanced genus as a new approach to studying balanced triangulations, which we expect to be a useful tool for understanding their structure.
As shown in \cite{CP, Gag, Gross, Stahl}, for any cyclic permutation  
$\varepsilon$ of $[d]$, $\Lambda(\Delta)$ admits a regular embedding  
$i_\varepsilon : \Lambda(\Delta) \hookrightarrow F_\varepsilon$,  
where $F_\varepsilon$ is orientable if $\Lambda(\Delta)$ is bipartite,  
and non-orientable otherwise. For a balanced normal $d$-pseudomanifold $\Delta$  with $d\geq 3$, the {\em balanced genus} $\mathcal G (\Delta)$ of $\Delta$ is defined as the smallest genus of the orientable surface (or half the genus, in the non-orientable case) into which $\Lambda(\Delta)$ embeds regularly. Similarly, the balanced genus $\mathcal {G}_M$ of a triangulable space $M$ is defined as the minimum balanced genus among all possible balanced triangulations of $M$. The primary focus of this article is to establish a lower bound for the balanced genus of $M$ in terms of its Euler characteristic and the rank of its fundamental group.

We establish lower bounds for the balanced genus of 3-manifolds and PL 4-manifolds and determine necessary and sufficient conditions for a 3-manifold or 4-manifold to be homeomorphic to the 3-sphere or 4-sphere, respectively. Furthermore, we compute the balanced genus for several well-known manifolds, describing its effectiveness in distinguishing different topological structures. Given a PL manifold $M$, denote by $m(M)$ the rank of the fundamental group of $M$. The main results of this article are the following:

\begin{Theorem}\label{main1}
Let $M$ be a $3$-manifold that is not homeomorphic to a sphere. Then  $\mathcal{G}_M\geq m(M)+3$. 
\end{Theorem}

\begin{Theorem}\label{main2}
Let $M$ be a PL $4$-manifold that is not homeomorphic to a sphere. Then  $\mathcal{G}_M\geq 2 \chi(M)+ 5m(M) +11$. 
\end{Theorem}

The sharpness of the inequalities in Theorems \ref{main1} and \ref{main2} is given by the sphere bundles over circles. Moreover, a given $3$-manifold $M$ is homeomorphic to the 3-sphere if and only if $\mathcal{G}_M \leq 3$, and a PL $4$-manifold $M$ is homeomorphic to the 4-sphere if and only if $\mathcal{G}_M \leq 2\chi (M)+10$.

We believe that the concept of balanced genus will open new avenues in topological combinatorics, providing novel insights into PL invariants, minimal triangulations, and classification problems in low-dimensional topology. Furthermore, the study of balanced genus raises several intriguing open questions. In the final section, we outline key research directions, including the classification of PL 4-manifolds based on balanced genus, generalizations to higher dimensions, and potential applications to the study of normal pseudomanifolds.

\section{Balanced complexes}
All the simplicial complexes considered in this article are finite, and their simplices are geometric. Moreover, we assume that $\emptyset$ is present in every simplicial complex as the only simplex of dimension -1. A simplicial complex is called {\em pure} if all its maximal simplices have the same dimension. The set of all vertices of a simplicial complex $\D$ is denoted by $V(\D)$, and the vertices contained in a simplex $\s$ are denoted by $V(\s)$. Any simplex contained in a simplicial complex $\D$ is called a face of $\D$, and the maximal faces of $\D$ are called facets. The collection of all faces of dimension at most 1 in $\D$ is called the {\em graph} of $\D$ and is denoted by $G(\D)$. Given a simplex $\s$, if $\t$ is a simplex obtained as the convex hull of a subset of the vertices of $\s$, then $\t$ is called a face of $\s$ and is denoted by $\t\leq\s$. For any set $S$, the cardinality of $S$ is denoted by $|S|$.

The join of two simplices $\s$ and $\t$ is the simplex obtained as the convex hull of vertices in $V(\s)\cup V(\t)$, denoted by $\s\t$ (or $\s\star\t$). Two simplicial complexes $\D_1$ and $\D_2$ are said to be independent if, for every pair of simplices $\s\in\D_1$ and $\t\in\D_2$ of dimensions $p$ and $q$, respectively, their join $\s\t$ is a simplex of dimension $p+q+1$. The join of two independent simplicial complexes $\D_1$ and $\D_2$ is defined as the simplicial complex $\D_1\star\D_2:=\{\s\t : \s\in\D_1,\t\in\D_2\}$. For a pair $(\s,\Delta)$, where $\s$ is a simplex and $\D$ is a simplicial complex, $\s\star\D$ represents the simplicial complex $\{\alpha:\alpha\leq\s\}\star\D$. 

The {\em link} of any face $\sigma$ in $\D$ is defined as the collection $\{ \gamma\in \D : \gamma\cap\sigma=\emptyset$ and $ \gamma\sigma\in \D\}$, denoted by $\lk (\sigma,\D)$. The {\em star} of $\sigma$ is defined as $\{\alpha : \alpha\leq\sigma \beta, \beta\in \lk (\sigma,\D)\}$, denoted by $\st (\sigma,\D)$. For any simplex $\s$ in a simplicial complex $\D$, the number of vertices in $\lk (\s,\D)$ is called the degree of the simplex $\s$, denoted by $\deg(\s,\D)$ or simply $\deg(\s)$ if the underlying simplicial complex is clear from the context.

Let $\D$ be a pure simplicial complex of dimension $d$. The collection of all simplices in $\D$, together with the subspace topology induced from $\mathbb{R}^m$ for some $m$, is called the geometric carrier of $\D$, denoted by $\|\D\|$. We say that $\D$ is a {\em normal $d$-pseudomanifold} if:  $(i)$ Every $(d-1)$-face of $\D$ is contained in exactly two facets of $\D$. $(ii)$ The link of every simplex of codimension two or more is connected. Throughout the article, we use $[d]=\{0,1,\dots, d\}$, and we denote the rank of the fundamental group of $\D$ by $m(\D)$.

\begin{definition}
{\rm Let $\D$ be a $d$-dimensional simplicial complex. We say that $\D$ is balanced if there exists a partition $(V_0,\dots, V_{d})$ of $V(\D)$, the set of vertices of $\D$, such that $|V(\s)\cap V_i|\leq 1$ for every simplex $\s\in\D$ and for every $V_i$ of the partition.} 
\end{definition}

The definition of a balanced simplicial complex given above can be interpreted in terms of vertex coloring of $\D$, where each vertex in the set $V_i$ is assigned the color $i$. In other words, the vertices of the simplicial complex $\D$ are equipped with a fixed vertex coloring $\kappa : V(\D) \to [d]$, such that $\kappa(v) = i$ for every vertex $v \in V_i$ and every $i \in [d]$. We denote such a balanced complex by $(\D, \kappa)$ (or simply $\D$ when $\kappa$ is understood). A normal $d$-pseudomanifold that is also balanced is called a {\em balanced normal $d$-pseudomanifold}. Similarly, a PL $d$-manifold that is also balanced is called a {\em balanced $d$-manifold}.

 The well-known combinatorial operations on simplicial complexes, such as the connected sum, can be adapted to the class of balanced, pure simplicial complexes by introducing additional constraints on the coloring of identified vertices. Let $\s_1$ and $\s_2$ be two facets of pure balanced simplicial complexes of the same dimension, say $\D_1$ and $\D_2$, respectively. A bijection $\psi: \s_1\to\s_2$ is said to be admissible if, for every vertex $x$ in $\s_1$, $\kappa(x)=\kappa(\psi(x))$. If $\psi$ is an admissible bijection between $\s_1$ and $\s_2$, then we can construct a new simplicial complex $\D_1\#_{\psi}\D_2$ by identifying all faces $\tau\leq \s_1$ with  $\psi(\tau)\leq\s_2$ and then removing the identified facet. We say that $\D_1\#_{\psi}\D_2$ is a {\em balanced connected sum} of $\D_1$ and $\D_2$.

 In \cite{Alexander}, it was established that two simplicial complexes are PL homeomorphic if and only if they are stellar equivalent. On the other hand, Pachner \cite{Pancher} showed that two closed PL manifolds are PL homeomorphic if and only if they are bistellar equivalent. The balanced analog of these results is described in terms of the {\em cross-flip} operation (cf. \cite{IzmestievKleeNovik}). Specifically, two balanced manifolds are PL homeomorphic if and only if there exists a sequence of cross-flips transforming one into the other \cite{IzmestievKleeNovik}. In \cite{KV}, the authors proved that two balanced manifolds with isomorphic boundaries are PL homeomorphic if and only if they are related by a sequence of cross-flips.  


One of the fundamental enumerative invariants of a $d$-dimensional simplicial complex $\D$ is its $f$-vector $(f_{-1}(\D),f_0(\D),\dots, f_d(\D))$, where $f_i(\D)$ denotes the number of $i$-simplices in $\D$ for $-1\leq i\leq d$, with $\emptyset$ considered the only simplex of dimension $-1$ in every simplicial complex. In the study of the $f$-vector and its properties in a simplicial complex $\D$, another associated vector, called the $h$-vector, $(h_0(\D), h_1(\D), \dots, h_{d+1}(\D))$, is commonly used. The $h$-vector is defined such that the $h$-polynomial satisfyes $\sum_{i=0}^{d+1} h_i(\D) x^{d-i}=\sum_{i=0}^{d+1}f_{i-1}(\D) (x-1)^{d+1-i}$. Equivalently, $$ h_i(\D)=\sum_{j=0}^{i}(-1)^{i-j} \binom{d+1-j}{i-j}f_{j-1}(\D).$$
\noindent We will use the notations $f_i$ and $h_i$ instead of $f_i(\D)$ and $h_i(\D)$, respectively, when the underlying simplicial complex is clear from the context.
 
The representation of balanced complexes in terms of vertex colors allows us to redefine the notion of $f$- and $h$-vectors for balanced simplicial complexes. Let $\D$ be a balanced simplicial complex of dimension $d$. For any subset of colors $S\subseteq [d]$, define $\D_S$ as the subcomplex of $\D$ consisting of all simplices $\s$ such that $\kappa(V(\s))\subseteq S$. We also define $f_{S}(\D)$ (or simply $f_{S})$) as the the number of faces in $\D$ with $\kappa (V(\s))=S$. The subcomplex $\D_S$ is known as the $S$-rank-selected subcomplex (or simply rank-selected subcomplex) of $\D$ in the literature. We use the notation $f^{ij}_{d-2}$ to denote $f_{[d]\setminus S}$, where $S=\{i,j\}\subseteq [d]$. The number $f_S$ (or $f_S(\D)$) is called the {\em flag $f$-number} of $\D$ \wrt the color set $S$, and the collection $(f_S)_{S\subseteq [d]}$ is called the {\em flag $f$-vector} of $\D$. The flag $f$-vector provides a refinement of the $f$-vector in the following sense: 
$$f_{i-1}=\sum_{S\subseteq [d], |S|=i}f_{S} \quad \text{for}\hspace{0.3cm} 0\leq i\leq d.$$  
By convention, we take $f_{\emptyset}=1$. Similarly, the \textit{ flag $h$-number} corresponding to the color set $T\subseteq [d]$, denoted by $h_T$ (or $h_T(\D)$), is defined as: $$h_{T}=\sum_{S\subseteq T} (-1)^{|T|-|S|}f_{S},$$
and the \textit{flag $h$-vector} of $\D$ is the collection $(h_S)_{S\subseteq [d]}$. The relationship between the $h$-vector and the flag $h$-vector of a balanced simplicial complex is given by the following \cite{Stanley}: 

$$h_j=\sum_{T\subseteq [d], |T|=j } h_{T}.$$

A pure simplicial complex of dimension $d$ is called a {\em semi-Eulerian complex} if, for every simplex $\s$ of dimension $i\geq 0$ in $\D$, the Euler characteristic of its link is the same as that of the $(d-i-1)$-dimensional sphere, i.e., $\chi(\lk (\s,\D))=\chi(\mathbb{S}^{d-i-1})$. A semi-Eulerian complex is called {\em Eulerian} if $\chi(\D)= \chi(\mathbb{S}^{d})$. A PL manifold is a typical example of a semi-Eulerian complex.

\begin{Lemma}{\rm \cite{Swartz2009}}\label{Swartz}
Let $\D$ be a balanced semi-Eulerian complex of dimension $d$. Then for every subset $S\subseteq [d]$, we have: $$h_{[d]-S} - h_{S}= (-1)^{|S|}[\chi(\D)-\chi(\mathbb{S}^{d})].$$ 
\end{Lemma}

\begin{Corollary}\label{equality of h values}
If $\D$ is a balanced $3$-manifold, then $h_{S}=h_{[3]-S}$ for every subset $S\subseteq [3]$ with $|S|=2$.
\end{Corollary}

\section{Balanced genus}
Let $\D$ be a balanced  normal $d$-pseudomanifold, where $d\geq 3$. Note that $\D$ is orientable (resp. non-orientable) if and only if the dual graph $\Lambda(\Delta)$ is bipartite (resp., non-bipartite). It is known (see \cite{CP, Gag, Gross, Stahl}) that for every cyclic permutation  
$
\varepsilon = (\varepsilon_0, \varepsilon_1, \varepsilon_2, \dots, \varepsilon_d)$
of $[d]$, there exists a regular embedding  
$
i_\varepsilon : \Lambda(\Delta) \hookrightarrow F_\varepsilon,
$
where $F_\varepsilon$ is a closed orientable surface whenever $\Lambda(\Delta)$ is bipartite,  
and a closed non-orientable surface whenever $\Lambda(\Delta)$ is non-bipartite.
 Let $\varepsilon= (\varepsilon_0, \varepsilon_1, \varepsilon_2, \dots , \varepsilon_d)$ be a cyclic permutation of $[d]$, and let $i_\varepsilon : \Lambda(\Delta)\hookrightarrow F_\varepsilon$ be a regular embedding of $\Lambda(\D)$, where $F_{\varepsilon}$ is a closed surface. Observe that the regular embedding $i_\varepsilon$ provides a cell complex structure on the surface $F_{\varepsilon}$. The set of vertices and edges of this cell complex structure on $F_{\varepsilon}$, inherited from $i_\varepsilon$, is the same as that of $\Lambda(\Delta)$.

The number of vertices in $\Lambda(\Delta)$ is the same as $f_d(\D)$. Since  $\Lambda(\Delta)$ is a $(d+1)$-regular colored graph, the number of edges in $\Lambda(\Delta)$ is $\frac{(d+1) f_d}{2}$. Furthermore, since the embedding $i_\varepsilon$ is regular, the number of faces in the cell complex structure of $F_{\varepsilon}$, inherited from $i_\varepsilon$, is equal to $\sum_{i \in \mathbb{Z}_{d+1}} C_{\varepsilon_i\varepsilon_{i+1}}$, where $C_{\varepsilon_i\varepsilon_{i+1}}$ represents the number of bi-colored cycles in  $\Lambda(\Delta)$ with colors $\varepsilon_i$ and $\varepsilon_{i+1}$ (with addition modulo $d+1$). On the other hand, each $(d-2)$-simplex in $\Delta$ with colors $[d]\setminus \{\varepsilon_i,\varepsilon_{i+1}\}$ corresponds uniquely to a bi-colored cycle in $\Lambda(\Delta)$ with colors $\varepsilon_i$ and $\varepsilon_{i+1}$.
Therefore,
$$\chi (F_\varepsilon)= \sum_{i \in \mathbb{Z}_{d+1}}f_{d-2}^{\varepsilon_i\varepsilon_{i+1}} + (1-d)  
\frac{f_d}{2},$$
where $f_{d-2}^{\varepsilon_i\varepsilon_{i+1}}$ denotes the number of $(d-2)$-simplices in $\Delta$ with colors $[d]\setminus \{\varepsilon_i,\varepsilon_{i+1}\}$.
In the orientable (resp. non-orientable) case, the integer  
\begin{equation}\label{rho1}
\rho_{\varepsilon}(\Delta) := 1 - \frac{\chi (F_\varepsilon)}{2}
\end{equation}
equals the genus (resp. half of the genus) of the surface $F_{\varepsilon}$, and we refer to it as the \textit{balanced $\e$-genus} of $\D$ corresponding to the cyclic permutation $\e$ of $[d]$. Therefore, the balance genus of $\Delta$ is given by
$$\mathcal G(\Delta)= \min \{\rho_{\varepsilon}(\Delta) \ : \  \varepsilon \ \text{ is a cyclic permutation of } \ [d]\}.$$
If $M$ is a PL manifold, then the balanced genus of $M$ is given by, $$\mathcal G{_M}= \min \{\mathcal G(\Delta) \ : \D \text{ is a balanced triangulation of} \ M\}.$$

Let $\D$ be a balanced normal $d$-pseudomanifold. For each cyclic permutation $\varepsilon=(\e_0,\e_1,\dots,\e_d)$ of $[d]$, the balanced $\e$-genus of $\D$ corresponding to $\e$, as defined in equation \eqref{rho1}, can be expressed as:

\begin{equation}\label{rho2}
 \rho_{\varepsilon}(\D)=1-\frac{1-d}{4}f_{d}-\frac{1}{2}\sum_{i\in\mathbb{Z}_{d+1}}f_{d-2}^{\e_i\e_{i+1}}.
\end{equation} 

Note that if $\D$ is an orientable balanced normal $d$-pseudomanifold, then $\mathcal{G}(\Delta)$ is a non-negative integer. However, if $\D$ is a non-orientable balanced normal $d$-pseudomanifold, then $\mathcal{G}(\Delta)=\frac{n}{2}$, where $n$ is a non-negative integer.

 \begin{Remark}
     {\rm Let $\D_1$ and $\D_2$ be two balanced normal $d$-pseudomanifolds, and let $\D_1\#\D_2$ be a balanced connected sum of $\D_1$ and $\D_2$. Then, $f_{d} (\D_1\#\D_2)=f_{d} (\D_1)+f_{d} (\D_2)-2$, and for any two-element subset $\{i,j\}$ of $[d]$, we have $f^{ij}_{d-2} (\D_1\#\D_2)=f^{ij}_{d-2} (\D_1)+f^{ij}_{d-2} (\D_2)-1$. Therefore, for every cyclic permutation $\e$ of $[d]$, we have 
     \begin{eqnarray*}
\rho_{\varepsilon}(\D_1\#\D_2)&=& 1-\frac{1-d}{4}(f_{d} (\D_1)+f_{d} (\D_2)-2)-\frac{1}{2}\sum_{i\in\mathbb{Z}_{d+1}}(f^{\e_i\e_{i+1}}_{d-2} (\D_1)+f^{\e_i\e_{i+1}}_{d-2} (\D_2)-1)\\
&=& \rho_{\varepsilon}(\D_1)+\rho_{\varepsilon}(\D_2)+\frac{1-d}{2}+\frac{d+1}{2}\\
&=& \rho_{\varepsilon}(\D_1)+\rho_{\varepsilon}(\D_2).
\end{eqnarray*}
     }
 \end{Remark}
 
\begin{Lemma}\label{d dim inequality}
Let $\D$ be a balanced normal $d$-pseudomanifold, where $d\geq 3$. Then $\mathcal{G}(\D)\geq 1+\frac{d-3}{8}f_d$.
\end{Lemma}
\begin{proof}
Let $\s$ be a $(d-2)$-simplex in $\D$. Since $\D$ is balanced, $\lk (\s,\D)$ contains at least four edges. If $\s_1$ is any other $(d-2)$-simplex such that $\kappa(V(\s))=\kappa(V(\s_1))$, then there are no common facets between $\st (\s,\D)$ and $\st (\s_1,\D)$. Let $[d]\setminus \kappa(V(\s))=\{i,j\}$. Since the link of each $(d-2)$-simplex in $\D$ contains at least four edges of color $\{i,j\}$, we have $f_d\geq 4 f_{d-2}^{ij}$. Therefore, for every cyclic permutation $\varepsilon$ of $[d]$, we have $\rho_{\varepsilon}(\D)\geq 1+\frac{d-1}{4}f_d-\frac{d+1}{8}f_d = 1+\frac{d-3}{8}f_d$. This completes the proof. 
\end{proof}


\begin{Lemma}\label{minimum balanced genus of manifold}
Let $d\geq 3$, and let $\Delta$ be a balanced normal $d$-pseudomanifold. Then $\mathcal{G}(\Delta)\;\geq\; 1+(d-3)2^{d-2}$. In particular, for a PL $d$-manifold $M$, we have $\mathcal{G}_M \;\geq\; 1+(d-3)2^{d-2},$ and equality is attained when $M$ is the $d$-sphere.
\end{Lemma}

\begin{proof}
Since $\Delta$ is balanced, its vertices admit a proper $(d+1)$-coloring with at least two vertices in each color class. Let $v,w$ be two vertices of the same color. The link of a vertex in $\Delta$ is a balanced normal $(d-1)$-pseudomanifold, so by induction it has at least $2^d$ facets. Hence, $\Delta$ has at least $2^d$ facets containing $v$ and at least $2^d$ facets containing $w$. These sets are disjoint, since a facet cannot contain two vertices of the same color. Therefore, $\Delta$ has at least $2^{d+1}$ facets. Consequently, Lemma \ref{d dim inequality} implies

$$
\mathcal{G}(\Delta)\;\geq\; 1+(d-3)2^{d-2}.
$$

For a PL $d$-manifold $M$, the same bound holds.
To establish the equality, consider the balanced triangulation of $\mathbb{S}^d$, known as the octahedral $d$-sphere, given by $\D=\p(a_0b_0)\star\cdots\star\p(a_db_d)$. This complex satisfies $f_d(\D)=2^{d+1}$ and $f^{ij}_{d-2}(\D)=2^{d-1}$ for every two-element subset $\{i,j\}\subseteq [d]$. Therefore, from the relation mentioned in \eqref{rho2}, we have $\rho_{\e}(\D)=1+ (d-3)2^{d-2}$ for every cyclic permutation $\e$ of $[d]$. Hence, for $d\geq 3$, we have $\mathcal G{_{\mathbb{S}^d}}=1+ (d-3) 2^{d-2}$.
\end{proof}

\subsection{Balanced genus in low-dimensions}
In the case of dimension 3, the balanced $\e$-genus is given by
$$   \rho_{\varepsilon}(\D)=1+\frac{1}{2}[f_3-\sum_{i\in\mathbb{Z}_{4}}f_{1}^{\e_i\e_{i+1}}].$$
The Dehn-Sommerville equation asserts that if $\D$ is a triangulated 3-manifold, then $f_3=f_1-f_0$. Therefore, the balanced $\varepsilon$-genus of a balanced 3-manifold $\D$ is given by $$\rho_{\varepsilon}(\D)= 1+\frac{1}{2}[f_1-f_0-\sum_{i\in\mathbb{Z}_{4}}f_{1}^{\e_i\e_{i+1}}].$$ Note that with the permutation $\varepsilon:=(\e_0,\e_1,\e_2,\e_3)$, the total number of edges $f_1$ of $\D$ can be expressed as $$f_1=f_{\{\e_0,e_1\}}+f_{\{\e_1,\e_2\}}+f_{\{\e_2,\e_3\}}+f_{\{\e_3,\e_0\}}+f_{\{\e_0,\e_2\}}+f_{\{\e_1,\e_3\}}.$$ Therefore, $f_1-\sum_{i\in\mathbb{Z}_{4}}f_{1}^{\e_i\e_{i+1}}=f_{\{\e_0,\e_2\}}+f_{\{\e_1,\e_3\}}$, and the balanced $\e$-genus is given by $$\rho_{\varepsilon}(\D)= 1+\frac{1}{2}[f_{\{\e_0,\e_2\}}+f_{\{\e_1,\e_3\}}-f_0].$$ It follows from Corollary \ref{equality of h values} that $h_{\{\e_0,\e_2\}}=h_{\{\e_1,\e_3\}}$, which implies that $f_{\{\e_0,\e_2\}}-f_{\{\e_0\}}-f_{\{\e_2\}}=f_{\{\e_1,\e_3\}}-f_{\{\e_0\}}-f_{\{\e_2\}}$. Thus, if $\D$ is a balanced 3-manifold, then for a cyclic permutation $\e=(\e_0,\e_1,\e_2,\e_3)$, the balanced $\e$-genus is given by

\begin{equation}\label{3dim rho}
 \rho_{\varepsilon}(\D)= 1+ f_{\{\e_0,\e_2\}}-f_{\{\e_0\}}-f_{\{\e_2\}}. 
\end{equation}

Let $\D$ be a balanced $4$-manifold. Then, for a cyclic permutation $\varepsilon$ of $[4]$, the balanced $\varepsilon$-genus $\rho_\e$ is given by the following expression:

\begin{equation}\label{4dim rho1}
 \rho_{\varepsilon}(\D)=1+\frac{3}{4}f_4-\frac{1}{2}.\sum_{i\in\mathbb{Z}_{5}}f_{2}^{\e_i\e_{i+1}}.
\end{equation}
Now, Lemma \ref{Swartz} implies that $h_{[4]\setminus\{\e_i,\e_{i+1}\}}-h_{\{\e_i,\e_{i+1}\}}=\chi(\D)-2$. Therefore, for a cyclic permutation $\e=(\e_0,\dots,\e_4)$, we have:
\begin{equation}\label{f012 in 4-dim}
 f_{\{\e_0,\e_1,\e_2\}}-f_{\{\e_0,\e_1\}}-f_{\{\e_1,\e_2\}}-f_{\{\e_2,\e_0\}}-f_{\{\e_3,\e_4\}}+f_0=\chi(\D).
\end{equation}
The Dehn-Sommerville equation for Euler $4$-manifolds states that $f_4(\D)=2f_1(\D)-6f_0(\D)+6\chi(\D)$. Therefore, using equation \eqref{f012 in 4-dim} and summing  over all five 3-cycles appearing in the formula of $\rho_{\e}(\D)$, we obtain:

\begin{eqnarray*}
\rho_{\varepsilon}(\D)&=& 1+\frac{3}{4}f_4+\frac{1}{2}(5f_0-f_1-2\sum_{i\in\mathbb{Z}_{5}}f_{\{\e_i\e_{i+1}\}}-5\chi(\D))\\
&=& 1+\frac{1}{2}(2f_1-4f_0+4\chi(\D)- 2\sum_{i\in\mathbb{Z}_{5}}f_{\{\e_i\e_{i+1}\}})\\
&=& 1+2 \chi(\D)+ f_1- \sum_{i\in\mathbb{Z}_{5}}f_{\{\e_i\e_{i+1}\}}-2f_0.
\end{eqnarray*}
Thus, if $\D$ is a balanced triangulation of a PL $4$-manifold, then 
\begin{equation}\label{4dim rho2}
 \rho_{\varepsilon}(\D)= 1+2 \chi(\D)+ f_1(\D)- \sum_{i\in\mathbb{Z}_{5}}f_{\{\e_i\e_{i+1}\}}(\D)-2f_0(\D).
\end{equation}
It is important to note that, from the expressions in \eqref{3dim rho} and \eqref{4dim rho2}, if $M$ is a PL $d$-manifold with $d = 3$ or $4$, then $\mathcal{G}_M$ is a non-negative integer. Moreover, it was proved in \cite{CP} that for non-orientable PL $d$-manifolds (with $d \geq 3$), the surface $F_{\varepsilon}$ has even genus, which implies that $\mathcal{G}_M$ is always an integer.

\section{The edge-path group}
In this section, we review the concept of the fundamental group on simplicial complexes, as described in the book by E. Spanier \cite{Spanier}. Let $\D$ be a simplicial complex. As previously defined, a 1-simplex $e$ is also called an edge that can be considered as an ordered pair $(v,v')$, where $v$ and $v'$ are two vertices in $\D$. The first vertex, $v$, is called the origin of the edge, denoted as Orig($e$), and the second vertex, $v'$, is called the end of the edge, denoted as End($e$). An edge-path $\zeta$ of length $r$ in $\D$ is a non-empty sequence of edges $e_1e_2\cdots e_r$ such that End($e_i$) = Orig($e_{i+1}$) for $1\leq i\leq r-1$. The origin and the end of the path $\zeta$ are defined as  Orig $\zeta$ = Orig($e_1$) and End $\zeta$ = End($e_r$), respectively. A closed edge-path at a vertex $v_0$ is an edge-path $\zeta$ such that Orig $\zeta$ = $v_0$ = End $\zeta$.

Let $\z_1$ and $\z_2$ be two edge-paths in $\D$, with End $\z_1$ = Orig $\z_2$. Then, the product edge-path $\z_1\z_2$ is defined as the edge-path consisting of the sequence of edges in $\z_1$ followed by the sequence of edges in $\z_2$. Therefore, Orig $\z_1\z_2$ = Orig $\z_1$ and End $\z_1\z_2$ = End $\z_2$. Two edge-paths $\z_1$ and $\z_2$ in $\D$ are called \textit{simply equivalent} if there exist vertices $v_1,v_2,$ and $v_3$ in $\D$ such that $v_1v_2v_3$ is a simplex in $\D$, and the unordered pair $\{\z_1,\z_2\}$ equals one of the following: 

\begin{enumerate}[$(i)$]
\item The unordered pair $\{(v_1,v_3), (v_1,v_2)(v_2,v_3)\}$.
\item The unordered pair $\{\z(v_1,v_3), \z(v_1,v_2)(v_2,v_3)\}$ for some edge-path $\z$ in $\D$ with End $\z = v_1$.
\item The unordered pair $\{(v_1,v_3)\z', (v_1,v_2)(v_2,v_3)\z'\}$ for some edge-path $\z'$ in $\D$ with Orig $\z' = v_3$.
\item The unordered pair $\{\z(v_1,v_3)\z', \z(v_1,v_2)(v_2,v_3)\z'\}$ for some edge-paths $\z$ and $\z'$ in $\D$ with End $\z = v_1$, and Orig $\z' = v_3$.
\end{enumerate}

Two edge-paths $\z$ and $\z'$ are said to be \textit{equivalent}, denoted by $\z\sim \z'$, if there exists a finite sequence of edge-paths $\z^{(0)},\z^{(1)},\dots , \z^{(k)}$ such that $\z=\z^{(0)}, \z'=\z^{(k)},$ and $\z^{(i-1)}$ and $\z^{(i)}$ are simply equivalent for $1\leq i\leq k$. If  $\z\sim \z'$, then clearly Orig $\z$ = Orig $\z'$ and End $\z$ = End $\z'$. We denote $[\z]$ as the equivalence class of paths containing the edge-path $\z$.

Fix a vertex $v_0$ in $\D$. Let $E(\D,v_0)$ be the set of all equivalence classes $[\z]$, where $\z$ is a closed edge-path in $\D$ with Orig $\z = v_0 =$ End $\z$. Note that if $\z_1\sim \z_1^{'}$ and $\z_2\sim \z_2^{'}$, with Orig $\z_2 =$ End $\z_1$, then $\z_1\z_2\sim \z_1^{'}\z_2^{'}$. Thus, there is a well-defined binary operation on $E(\D,v_0)$, namely $[\z_1]\circ [\z_2]=[\z_1\z_2]$. The set $E(\D,v_0)$, together with the binary operation `$\circ$' forms a group called the \textit{edge-path group} of $\D$. By convention, for a vertex $v$, we take $[(v,v)] =1$.

\begin{Lemma}{\rm \cite{Spanier}}
Let $\D$ be a simplicial complex, and let $v_0$ be a vertex in $\D$. The fundamental group $\pi_{1}(\|\D\|,v_0)$ is isomorphic to the edge-path group $E(\D,v_0)$.
\end{Lemma}
Let $\D$ be a connected simplicial complex, and let $T$ be a spanning tree in $G(\D)$, the 1-skeleton or the graph of $\D$. Let $G_T$ be the group generated by the edges $(v,v')$ of $\D$ modulo the following relations:
\begin{enumerate}[$(i)$]
\item $(v,v')=1$ if $(v,v')$ is an edge in $T$,

\item $(v_1,v_2)(v_2,v_3)=(v_1,v_3)$ if $v_1,v_2$ and $v_3$ are vertices of a simplex in $\D$. 
\end{enumerate} 
As notation, a typical element in the group $G_T$, represented as $e_1e_2\cdots e_r$ modulo the relations as mentioned in $(i)$ and $(ii)$ for some edges $e_1, e_2,\dots, e_r$, is denoted by $[e_1e_2\cdots e_r]_T$. By convention, for a vertex $v$, we take $[(v,v)]_T =1$. Moreover, given a cycle $C:=C(v_0,v_1,\dots,v_r)$ in $G(\D)$, we denote by $[C]_T$ the element $[(v_0,v_1)\cdots (v_r,v_0)]_T$.

\begin{Theorem}{\rm \cite{Spanier}}\label{edge-graph and edge-group}
Let $\D$ be a connected simplicial complex, and let $T$ be a spanning tree in $G(\D)$, the graph of $\D$. Then the edge-path group $E(\D,v_0)$ is isomorphic to $G_T$.
\end{Theorem}

The isomorphism used in Theorem \ref{edge-graph and edge-group} is the function $\alpha :  E(\D,v_0)\to G_T$ given by $$\alpha([(v_0,v_1)\cdots (v_n,v_0)]) = [(v_0,v_1)\cdots(v_n,v_0)]_T,$$ where the element on the right-hand side represents the equivalence class defined by the operations in $ G_T$. The inverse of this map is defined on the generators of $G$ as follows: for each edge $(v,v')$, if $\z_v$ and $\z_{v'}$ are unique edge-paths from $v_0$ to $v$ and from $v_0$ to $v'$ along $T$, respectively. Then $\alpha^{-1}: G_{T}\to E(\D,v_0)$ is given by $$\alpha^{-1}([(v,v')]_T)=[\z_v (v,v')\z^{-1}_{v'}],$$
where $\z^{-1}_{v'}$ indicates the path $\z_{v'}$ with the opposite direction. 

Let $T'$ be another spanning tree of $G(\D)$. Using similar arguments, the group $G_{T'}$ is isomorphic to the edge-path group $E(\D,v_0)$. If $\beta :  E(\D,v_0)\to G_{T'}$ is an isomorphism, then the composition $\beta\circ\alpha^{-1}: G_T\to G_{T'}$ defines an isomorphism between $G_T$ and $G_{T'}$.

\begin{Lemma}\label{unique cycle Ce is in link of a vertex}
Let $\D$ be a connected simplicial complex, and let $T$ be a spanning tree in $G(\D)$, the graph of $\D$. If $e$ is an edge in $G(\D)\setminus T$ such that the unique cycle $C_e$ in $T\cup\{e\}$ is contained in the link of a vertex, then $[C_e]_T=1$.
\end{Lemma}
\begin{proof}
Let $C_e=C(v_0,\dots,v_r)$, and let $v$ be a vertex in $\D$ such that $C_e\subseteq\lk (v,\D)$. Then, for each $0\leq i\leq r$, the simplex $vv_iv_{i+1}$ is a face of $\D$, where $v_{r+1}=v_0$. Now, consider the edge-path $(v_0,v_1)(v_1,v_2)\cdots (v_r,v_0)$ in the edge-path group $E(\D,v_0)$. We obtain the following edge-path equivalences:
\begin{eqnarray*}
(v_0,v_1)(v_1,v_2)\cdots (v_r,v_0)
 &\sim & (v_0,v)(v,v_1)(v_1,v_2)\cdots (v_r,v_0)\\
 &\sim & (v_0,v)(v,v_0)\\
  &\sim & (v_0,v_0).
\end{eqnarray*}
Therefore, we have $[C_e]_T=[(v_0,v_1)(v_1,v_2)\cdots (v_r,v_0)]_T=[(v_0,v_0)]_T=1$, which completes the proof.
\end{proof}
%
\begin{Lemma}{\rm \cite{S.Klee}}
Let $\D$ be a balanced normal $d$-pseudomanifold, and let $S$ be a two-element subset of $[d]$. If $v_0\in\D_{S}$, then every class in $E(\D,v_0)$ can be represented by a closed edge-path in $\D_{S}$ that contains $v_0$.
\end{Lemma}
\begin{Lemma}\label{edge generators}{\rm \cite{S.Klee}}
Let $\D$ be a balanced normal $d$-pseudomanifold, and let $S$ be a two-element subset of $[d]$. Let $G(\D)$ denote the graph of $\D$. Then, for a spanning tree $T$ of $G(\D)$, the group $G_T$ is generated by the edges $(v,v')$ such that $vv'\in\D_S$.
\end{Lemma}

 \section{Balanced normal $d$-pseudomanifolds}
Let $\D$ be a balanced normal $d$-pseudomanifold, and let $S=\{p,q\}$ be a two-element subset of $[d]$. Then the rank-selected subcomplex $\D_S$ is connected, and each vertex in $\D_S$ has a degree of at least two. Define $\Gamma_{S}= \Gamma_{pq}:= f_{\{p,q\}}-f_{\{p\}}-f_{\{q\}}$. 

\begin{Lemma}\label{strong r-connected}
Let $\D$ be a balanced normal $d$-pseudomanifold, and let $S$ be an $r$-element subset of $[d]$, where $r\geq 2$. If $\s$ and $\t$ are two $(r-1)$-simplices in $\D_S$, then there exists a sequence of $r$-simplices $\s_1,\dots,\s_m$ such that $\s=\s_1,\t=\s_m,$ and $\s_i\cap\s_{i+1}$ is a $(r-2)$-simplex in $\D_S$.
\end{Lemma}
\begin{proof}
Let $\s'$ and $\t'$ be two $d$-simplices in $\D$ such that $\s\leq\s'$ and $\t\leq\t'$. Then, by definition, there exists a sequence of $d$-simplices $\t_1,\dots,\t_m$ such that $\s'=\t_1$, $\t'=\t_m$, and $\t_i\cap\t_{i+1}$ is a simplex of dimension $d-1$. Therefore, $\t_1\restriction_S=\s$, $\t_m \restriction_S=\t$, and $|\kappa(\t_i\cap\t_{i+1})|\leq d$. Furthermore, $\t_i\cap\t_{i+1}\restriction_S$ is either an $r$-simplex or an $(r-1)$-simplex. 

Let $\s_1=\s=\t_1\restriction_S$ and $\s_2=\t_2\restriction_S$. Since $\t_1\restriction_S\cap \t_2\restriction_S=(\t_1\cap\t_2)\restriction_S$ is a simplex of dimension $r$ or $r-1$, it follows that either $\s_1=\s_2$ or $\s_1\cap\s_2$ is a simplex of dimension $r-1$. Now, apply the same procedure for each $i$, defining $\s_i=\t_i\restriction_S$. Observe that either $\s_i=\s_{i+1}$ or $\s_i\cap\s_{i+1}$ is a simplex of dimension $r-1$. Moreover, in the $m$-th step, we have $\s_m=\t_m\restriction_S=\t$. Thus, we obtain a sequence $\s_{i_1},\dots,\s_{i_k}$ such that $\s=\s_{i_1}$, $\t=\s_{i_k}$, and $\s_{i_j}\cap\s_{i_{j+1}}$ is a simplex of dimension $r-1$ in $\D_S$.
\end{proof}
\begin{Lemma}\label{all edge in one vertex link}
Let $\D$ be a balanced normal $d$-pseudomanifold, and let $S$ be a two-element subset of $[d]$. If there exists a vertex $u$ in $\D$ such that $\D_S\subseteq \lk (u,\D)$, then $m(\D)=0$.
\end{Lemma}
\begin{proof}
Fix a vertex $v_0$, and consider the edge-path group $E(\D,v_0)$. Let $\bar{T}$ be a spanning tree of $\D_S$, and let $T$ be a spanning tree of $G(\D)$ that extends $\bar{T}$. We will prove that the group $G_T$ is trivial. By Lemma \ref{edge generators}, the group $G_T$ is generated by the edges in $\D_S$. Let $\alpha:E(\D,v_0)\to G_T$ be the isomorphism mentioned earlier, and let $(v,v')$ be an edge in $\D_S$. Then $\alpha^{-1}([(v,v')]_T)=[\z_v(v,v')\z_{v'}^{-1}]$, where $\z_v$ and $\z_{v'}$ are paths in $\bar{T}$ from $v_0$ to $v$ and $v'$, respectively. 

Let $\z_v(v,v')\z_{v'}^{-1}=(v_0,v_1)(v_1,v_2)\dots (v_r,v)(v,v')(v',v_{r+1})\dots (v_{n-1},v_n)(v_n,v_0)$. Since $\D_{S}\subseteq\lk (u,\D)$, we have the following equivalences:
\begin{eqnarray*}
\z_v(v,v')\z_{v'}^{-1}&=&(v_0,v_1)(v_1,v_2)\dots (v_r,v)(v,v')(v',v_{r+1})\dots (v_{n-1},v_n)(v_n,v_0)\\
 &\sim & (v_0,u)(u,v_1)(v_1,v_2)\dots (v_r,v)(v,v')(v',v_{r+1})\dots (v_{n-1},v_n)(v_n,v_0)\\
 &\sim & (v_0,u)(u,v)(v,v')(v',v_{r+1})\dots (v_{n-1},v_n)(v_n,v_0)\\
 &\sim & (v_0,u)(u,v_n)(v_n,v_0)\\
 &\sim & (v_0,u)(u,v_0)\\
 &\sim & (v_0,v_0).
\end{eqnarray*}
Therefore, $\alpha^{-1}([(v,v')]_T)=[(v_0,v_0)]=1$. Since $\alpha$ is an isomorphism, it follows that  $[(v,v')]_T=1$ in $G_T$. Thus, $G_T$ is the trivial group, and this completes the proof.
\end{proof}
  
\begin{Lemma}\label{d-dim manifold with SS=0}
Let $\D$ be a balanced $d$-manifold, and let $S$ be a two-element subset of $[d]$. If $\Gamma_{S} = 0$, then $\|\D\|\cong \mathbb{S}^d$.
\end{Lemma}
\begin{proof}
Note that the degree of each vertex in $\D_S$ is at least two. Since $\Gamma_{S} = 0$, $\D_S$ is a cycle. Without loss of generality, let $S=\{d-1,d\}$, and consider the subcomplex $\D_{[d-2]}$. Let $\s$ be a $(d-2)$-simplex in $\D_{[d-2]}$. Then, $\lk (\s,\D)$ is a cycle contained in $\D_S$. Since $\D_S$ itself is a cycle, it follows that $\lk (\s,\D) = \D_S$ for every $(d-2)$-simplex $\s$ in $\D_{[d-2]}$. Thus, if $e$ is an edge in $\D_S$, then $\D_{[d-2]}\subseteq\lk (e,\D)\subseteq \D_{[d-2]}$, implying that $\lk (e,\D)=\D_{[d-2]}$ for every edge $e\in  \D_S$. Hence, we conclude that $\D_{[d-2]}$ is a balanced $(d-2)$-sphere, and $\D= \D_{[d-2]}\star \D_S$. This completes the proof.
\end{proof}

\begin{Lemma}\label{d-dim npm with degree of (d-3) simplex geq 3}
Let $\D$ be a balanced normal $d$-pseudomanifold, and let $S$ be a two-element subset of $[d]$ such that $\Gamma_{S}\geq 1$. Then, there exists a $(d-3)$-simplex $\s$ in $\D_{[d]\setminus S}$ such that $\deg(\s,\D_{[d]\setminus S})\geq 3$.
\end{Lemma}

\begin{proof}
If possible, assume that the degree of every $(d-3)$-simplex in $\D_{[d]\setminus S}$ is 2, i.e., each $(d-3)$-simplex is contained in exactly two facets of  $\D_{[d]\setminus S}$. We will prove that $\D_{[d]\setminus S}$  is a normal $(d-2)$-pseudomanifold for every $d\geq 3$. If $d=3$, then $\D_{[d]\setminus S}$ is a circle, and hence, we are done. 

Now, assume that $d\geq 4$. It follows from Lemma \ref{strong r-connected} that for any two $(d-2)$-simplices  $\s_1$ and $\s_2$ in $\D_{[d]\setminus S}$, there exists a sequence of $(d-2)$-simplices $\t_1,\dots,\t_s$ such that $\s_1=\t_1,\s_2=\t_s,$ and $\t_i\cap\t_{i+1}$ is a $(d-3)$-simplex in $\D_{[d]\setminus S}$. To show that $\D_{[d]\setminus S}$ is a normal $(d-2)$-pseudomanifold, we need to prove that the link of every simplex of codimension two or more in $\D_{[d]\setminus S}$ is connected.

Let $\t$ be a simplex of codimension two or more in $\D_{[d]\setminus S}$. Then $\t$ has codimension four or more in $\D$. Since $\lk (\t,\D)$ is a balanced normal pseudomanifold, $\lk (\t,\D)_{[d]\setminus S}$ is connected. Therefore, the link of every simplex of codimension two or more in  $\D_{[d]\setminus S}$ is connected. Hence,  $\D_{[d]\setminus S}$  is a normal $(d-2)$-pseudomanifold for every $d\geq 3$.

Let $e$ be an edge in $\D_S$. Then $\lk (e,\D)$ is a normal pseudomanifold of dimension $d-2$. Moreover, $\kappa (\lk (e,\D)) = [d]\setminus S$, which implies that $\lk (e,\D)=\D_{[d]\setminus S}$ for every edge $e$ in $\D_S$. Since the link of every maximal simplex of $\D_{[d]\setminus S}$ in $\D$ is a circle, it follows that $\D_S$ is a circle. This implies that $\Gamma_{S}=0$, leading to a contradiction. Thus, there exists a $(d-3)$-simplex $\s$ in $\D_{[d]\setminus S}$ such that $\deg(\s,\D_{[d]\setminus S})\geq 3$.
\end{proof}

\begin{definition}
{\rm Let $G$ be a graph, and let $C\subseteq G$ be a cycle. We say that $C$ is  an {\em almost-induced cycle} in $G$ if $C$ contains a vertex, say $c$, such that $\deg(c,G)>2$, while every other vertex $v\in V(C)\setminus\{c\}$ satisfies $\deg(v,G)=2$.}
\end{definition}
\begin{Lemma}\label{no almost-induced}
Let $\D$ be a balanced triangulation of a normal $d$-pseudomanifold, and let $S$ be a two-element subset of $[d]$. Then, there does not exist an almost-induced cycle $C$ in $\D_{S}$. 
\end{Lemma}
\begin{proof}
If possible, let $C$ be an almost-induced cycle in $\D_{S}$. Let $c$ be the vertex in $C$ with $\deg(c,\D_{S}) >2$, and let $c_1$ and $c_2$ be the vertices in $C$ that are adjacent to $c$. Since $\lk (c_1,\D)$ is a balanced normal $(d-1)$-pseudomanifold and $\deg(c_1,\D_S)=2$, we have $\lk (c_1,\D)=\p(cx)\star\lk (cc_1,\D)$, where $x$ is a vertex in $C$ adjacent to $c_1$ in $\D_{S}$. Therefore, $\lk (c_1x,\D)=\lk (cc_1,\D)$. Using the fact that $\deg(v,\D_{S})=2$ for every vertex $v\in C\setminus\{c\}$, we conclude that $\lk (yz,\D)=\lk (cc_1,\D)$ for every edge $yz\in C$. 

Since $\lk (cc_1,\D)=\lk (cc_2,\D)$, we have $\p(c_1c_2)\star\lk (cc_1,\D)\subseteq \lk (c,\D)$.
Furthermore, since $\lk (c,\D)$ is a balanced normal $(d-1)$-pseudomanifold, its structure implies that $\lk (c,\D)=\p(c_1c_2)\star\lk (cc_1,\D)$. This contradicts the fact that $\deg(c,\D_S)>2$. Therefore, $\D_{S}$ does not contain any almost-induced cycles.
\end{proof}

\begin{Lemma}\label{d-dim npm with SS=1}
Let $\D$ be a balanced normal $d$-pseudomanifold, and let $S$ be a two-element subset of $[d]$ such that $\Gamma_{S} = 1$. Then, we have $m(\D)=0$.
\end{Lemma}
\begin{proof}
Let $S=\{p,q\}$. Here, $f_{\{p,q\}}=f_{\{p\}}+f_{\{q\}}+1$, and by construction, the degree of each vertex in $\D_{S}$ is at least 2. The relation $\sum_{v\in V(\D_S)} \deg(v,\D_{S})= 2|E(\D_S)|$, where $E(\D_S)$ denotes the set of edges in $\D_S$, implies that the graph $\D_{S}$ contains at least one vertex with degree greater than 2. Since $\D_{S}$ does not contain an almost-induced cycle (see Lemma \ref{no almost-induced}), the only possible degree sequence for these graphs is $(3,3,2,2,2,\dots, 2)$. Furthermore, the graph is a cycle with a diagonal path joining two of its vertices. Therefore, at least two $(d-2)$-simplices in $\D_{[d]\setminus S}$ have different links. By applying Lemma \ref{strong r-connected}, we can select two distinct $(d-2)$-simplices, $\s$ and $\t$, in $\D_{[d]\setminus S}$ such that $\s\cap\t\neq\emptyset$ and $\lk (\s,\D)\neq\lk (\t,\D)$. Let $x\in V(\s\cap\t)$. Then, $\lk (\s,\D)\cup\lk (\t,\D)\subseteq\lk (x,\D)$. Consequently, $\lk (x,\D)$ contains two different cycles, say $C_1$ and $C_2$, where $C_1\cup C_2=\D_S$.

Let $e_1\in E(C_1)\setminus E(C_2)$ and $e_2\in E(C_2)\setminus E(C_1)$ be two edges, and consider the paths $T_1=C_1\setminus\{e_1\}$ and $T_2=C_2\setminus\{e_2\}$. Define $T'=T_1\cup T_2$, which forms a spanning tree of $\D_S$. Now, extend the spanning tree $T'$ to a spanning tree $T$ of $G(\D)$. 

Since $C_1$ and $C_2$ are the unique cycles in $T\cup \{e_1\}$ and $T\cup \{e_2\}$ containing $e_1$ and $e_2$, respectively, and there exists a vertex $x$ such that $C_1\cup C_2\subseteq \lk (x,\D)$, we apply Lemma \ref{unique cycle Ce is in link of a vertex} to obtain $[e_i]_T=[C_i]_T=1$ in $G_T$ for $i\in \{1,2\}$. Consequently, there exist two edges, $e_1$ and $e_2$, in $E(\D_{S})\setminus E(T)$ such that $[e_i]_T=1$ in $G_T$. Hence, the number of generators of the group $G_T$ is at most $f_{\{p,q\}}-f_{\{p\}}-f_{\{q\}}+1-2=0$, i.e., $m(\D)=0$. 
\end{proof}

\begin{Lemma}\label{d-dim manifold with SS=1}
Let $\D$ be a balanced $d$-manifold. If $S$ is a two-element subset of $[d]$ such that $\Gamma_{S}=1$, then $\|\D\|\cong \mathbb{S}^d$.
\end{Lemma}
\begin{proof}
By the same argument as in Lemma \ref{d-dim npm with SS=1}, the graph $\D_S$ is a cycle with a diagonal path joining two of its vertices. Let $a$ and $b$ be two vertices in $\D_{S}$ of degree 3. Then there exist three distinct paths of length at least two in $\D_{S}$ between $a$ and $b$, say $P_1,P_2,$ and $P_3$, such that any two paths intersect only at the vertices $a$ and $b$. Let us take the paths as $P_1:=P(a,a_1,\dots,b_1,b)$, $P_2:=P(a,a_2,\dots,b_2,b)$, and $P_3:=P(a,a_3,\dots,b_3,b)$. Note that if the length of some paths $P_j$ is two, then $a_j=b_j$ in such cases. We establish that $\cup_{i=1}^{3} P_i\star\lk (aa_i,\D)$  is a $d$-dimensional ball with its boundary coned off, proving that $\D$ is a triangulation of the $d$-dimensional sphere.

 All three vertices $a_1,a_2,$ and $a_3$ are adjacent to $a$ in $\D_S$ and have the same color. Therefore, $\lk (aa_1,\D)$, $\lk (aa_2,\D)$, and $\lk (aa_3,\D)$ are three distinct spheres in $\D_{[d]\setminus S}$, and hence $\st (a,\D)=\cup_{i=1}^{3} aa_i\star\lk (aa_i,\D)$. Moreover, $\st (a,\D)$ is a $d$-dimensional ball with boundary $\cup_{i=1}^{3} a_i\star\lk (aa_i,\D)$. Since $\deg(x,\D_{S})=2$ for every vertex $x$ in $P_1\setminus\{a,b\}$, we have $\lk (aa_1,\D)=\lk (e_1,\D)$ for every edge $e_1$ in $P_1$. Similarly, $\lk (e_2,\D)=\lk (aa_2,\D)$ and $\lk (e_3,\D)=\lk (aa_3,\D)$ for every edge $e_2$ and $e_3$ in $P_2$ and $P_3$, respectively. This further implies that, $\D_{[d]\setminus S}=\cup_{i=1}^3 \lk (aa_i,\D)$, because if $\t$ is a simplex in $\D_{[d]\setminus S}\setminus\cup_{i=1}^3 \lk (aa_i,\D)$, then $\lk (\t,\D)$ contains an edge $e$ from $\D_S$, which contradicts the fact that $\lk (e,\D) = \lk (aa_j,\D)$ for some $j\in\{1,2,3\}$.

Let us take the path $P'_1:=P(a_1,\dots,b_1)$. If $a_1=b_1$, then $P'_1\star \lk (aa_1,\D)$ is the $(d-1)$-ball $a_1\star \lk (aa_1,\D)$. If $a_1\neq b_1$, $P'_1\star \lk (aa_1,\D)$ is a $d$-dimensional ball in $\D$ with boundary $\p(a_1b_1)\star\lk (aa_1,\D)$, and it is attached to $\st (a,\D)$ in $\D$ along the $(d-1)$-ball $a_1\star \lk (aa_1,\D)$. Therefore, $\st (a,\D)\cup P'_1\star \lk (aa_1,\D)$ is a $d$-dimensional ball with boundary $b_1\star\lk (aa_1,\D)\cup_{i=2}^3 a_i\star \lk (aa_i,\D)$. 

Now, consider the path  $P'_2:=P(a_2,\dots,b_2)$. If $a_2=b_2$, then $P'_2\star \lk (aa_2,\D)$ is the $(d-1)$-ball $a_2\star \lk (aa_2,\D)$. If $a_2\neq b_2$, $P'_2\star \lk (aa_2,\D)$ is a $d$-dimensional ball in $\D$ with boundary $\p(a_2b_2)\star\lk (aa_2,\D)$, and it is attached to $\st (a,\D)\cup P'_1\star \lk (aa_1,\D)$ in $\D$ along the $(d-1)$-ball $a_2\star \lk (aa_2,\D)$. Therefore, $\st (a,\D)\cup_{i=1}^{2} P'_i\star \lk (aa_i,\D)$ is a $d$-dimensional ball with boundary $\cup_{i=1}^{2} b_i\star\lk (aa_i,\D)\cup a_3\star\lk (aa_3,\D)$.

Finally, consider the path $P'_3:=P(a_3,\dots,b_3)$. If $a_3=b_3$, then $P'_3\star \lk (aa_3,\D)$ is the $(d-1)$-ball $a_3\star \lk (a_3,\D)$. If $a_3\neq b_3$, then $P'_3\star \lk (aa_3,\D)$ is a $d$-dimensional ball in $\D$ with boundary $\p(a_3b_3)\star\lk (aa_3,\D)$, and it is attached to $\st (a,\D)\cup_{i=1}^{2} P'_i\star \lk (aa_i,\D)$ in $\D$ along the $(d-1)$-ball $a_3\star \lk (aa_3,\D)$. Therefore, $\st (a,\D)\cup_{i=1}^{3} P'_i\star \lk (aa_i,\D)$, which is the same as $\cup_{i=1}^{3} (P_i-b)\star\lk (aa_i,\D)$, is a $d$-dimensional ball with boundary $\cup_{i=1}^{3} b_i\star\lk (aa_i,\D)$. To complete the proof, it suffices to establish that the boundary of the $d$-dimensional ball $\st (b,\D)$ is $\cup_{i=1}^{3} b_i\star\lk (aa_i,\D)$.

Since $b_1, b_2$, and $b_3$ are three vertices in $\D_S$ adjacent to $b$, they have the same color. Moreover, $bb_i$ is an edge in $P_i$ implies that $\lk (bb_i,\D)=\lk (aa_i,\D)$ for each $1\leq i\leq 3$. Therefore, from $\deg(b,\D_S)=3$ we have $\st (b,\D)=\cup_{i=1}^{3} bb_i\star\lk (aa_i,\D)$, and $\lk (b,\D)=\cup_{i=1}^{3} b_i\star\lk (aa_i,\D)$. Hence, $\st (b,\D)$ is a $d$-dimensional ball with the boundary $\cup_{i=1}^{3} b_i\star\lk (aa_i,\D)$. This completes the proof.
%
\end{proof}
\begin{Theorem}\label{d-dim balanced NPM with SS geq 2}
Let $\D$ be a balanced normal $d$-pseudomanifold. If $S$ is a two-element subset of $[d]$ such that $\Gamma_{S}\geq 2$, then $\Gamma_{S}\geq m(\D)+2$.
\end{Theorem}
\begin{proof}
Let $S=\{p,q\}$. By Lemma \ref{d-dim npm with degree of (d-3) simplex geq 3}, $\D_{[d]\setminus S}$ contains a $(d-3)$-simplex, say $\t$, such that $\deg(\t,\Delta[d]\setminus S)\geq 3$. Choose a $(d - 3)$-simplex $\sigma$ in $\Delta[d]\setminus S$ with the largest $\deg(\sigma,\Delta[d]\setminus S)$. Hence, it suffices to build a spanning tree $T$ of $G(\D)$, and find three edges $e_i$, $1\leq i\leq 3$, in $E(\D_S)\setminus E(T)$ such that $e_i\in \lk(\sigma, \Delta)$ and $[e_i]_T = 1 $ in $G_T$, proving that a generating set of $G_T$ contains at most $f_{\{p,q\}}-f_{\{p\}}-f_{\{q\}}+1-3$, i.e., $\Gamma_S- 2$ elements.  These will imply that the rank of the group $G_T$ is at most $\Gamma_S- 2$, implying that $\Gamma_S\geq m(\D)+ 2$. We will split into a few cases. 

\vspace{.25cm}
\noindent\textbf{Case 1.}
Let $\deg(\s,\D_{[d]\setminus S})\geq 4$. Then $\lk (\s,\D)$ contains at least four different cycles from the graph $\D_{S}$, noting that the graph $\lk (\s,\D)_{S}$ is connected.

Let $C_1$ and $C_2$ be two cycles in $\lk (\s,\D)_{S}$ such that $C_1\cap C_2\neq \emptyset$. Let $e_1\in E(C_1)\setminus E(C_2)$ and $e_2\in E(C_2)\setminus E(C_1)$ be two edges, and consider the paths $T_1=C_1\setminus\{e_1\}$ and $T_2=C_2\setminus\{e_2\}$. Define $T'=T_1\cup T_2$. 

\vspace{.25cm}
\noindent\textbf{Case 1a.}
Suppose that $T'$ is not a spanning tree of $C_1\cup C_2$, i.e., $T'$ contains some cycle. Then each cycle in $T'$ contains edges from both $C_1$ and $C_2$. We proceed by deleting an edge from each cycle appearing in $T'$ that also belongs to $C_2$. This process results in a subgraph of $T'$ that forms a spanning tree of $C_1 \cup C_2$. Let $e_k$ be the last edge deleted from $T'\cap C_2$ to obtain a spanning tree $T''$ of $C_1\cup C_2$. Define $C'_i$ as the unique cycles in $T''\cup\{e_i\}$ containing $e_i,$ where $i\in\{1,2,k\}$. Then, $C'_i\subseteq C_1\cup C_2\subseteq \lk (\s,\D)_{S}$. Finally, we extend the spanning tree $T''$ to a spanning tree $T$ of $G(\D)$. 

For $i\in\{1,2,k\}$, $C'_i$ is the unique cycle in $T\cup\{e_i\}$ containing $e_i$, and $C'_i\subseteq C_1\cup C_2\subseteq \lk (\s,\D)$.  Therefore, from Lemma \ref{unique cycle Ce is in link of a vertex}, we have $[e_i]_T=[C_i]_T=1$ in $G_T$ for $i\in\{1,2,k\}$. Consequently, there exist three edges, $e_1, e_2,$ and $e_k$, in $E(\D_{S})\setminus E(T)$ such that $[e_i]_T=1$ in $G_T$.

\vspace{.25cm}
\noindent\textbf{Case 1b.}
Assume that $T'$ is a spanning tree of $C_1\cup C_2$. Then $T'\cup\{e_1,e_2\}$, which is the same as $C_1\cup C_2$, contains at most three distinct cycles. Since $\deg(\s,\D_{[d]\setminus S})\geq 4$, there exists a cycle $C_3\subseteq \lk (\s,\D)_{S}$ such that $C_3\cap(C_1\cup C_2)\neq\emptyset$ and $E(C_3)\setminus E(C_1\cup C_2)\neq\emptyset$. Let $e_3\in E(C_3)\setminus E(C_1\cup C_2)$. Consider the path $T_3=C_3\setminus\{e_3\}$, and then define $T^{1}= T'\cup T_3$.


If $T^{1}$ is a spanning tree of $C_1\cup C_2\cup C_3$, then extending $T^1$ to a spanning tree of $G$ will serve our purpose. If $T^1$ is not a spanning tree of $C_1 \cup C_2 \cup C_3$, then each cycle appearing in $T^1$ contains edges from both $T'$ and $C_3$. We proceed by deleting an edge from each cycle appearing in $T^1$ that also belongs to $C_3$, and create a new graph $T'''$, which forms a spanning tree of $C_1\cup C_2\cup C_3$. Let $e_r$ be the last edge deleted from $T^1\cap C_3$ to obtain a spanning tree $T'''$ of $C_1\cup C_2\cup C_3$. Define $C'_i$ as the unique cycle in $T'''\cup\{e_i\}$, where $i\in\{1,2,r\}$. Then $C'_i\subseteq C_1\cup C_2\cup C_3\subseteq \lk (\s,\D)_{S}$. 

Therefore, there exists a spanning tree $T'''$ of $C_1\cup C_2\cup C_3$ such that the difference between the number of edges in $C_1\cup C_2\cup C_3$ and the number of edges in $T'''$ is at least three. Let us extend the spanning tree $T'''$ to a spanning tree $T$ of $G(\D)$. Since $C'_i$ is the unique cycle in $T\cup\{e_i\}$ containing $e_i,$ where $i\in\{1,2,r\}$, and $C'_i\subseteq C_1\cup C_2\subseteq \lk (\s,\D)$, using Lemma \ref{unique cycle Ce is in link of a vertex}, we have $[e_i]_T=[C_i]_T=1$ in $G_T$ for $i\in\{1,2,r\}$. Thus, we have three edges $e_1, e_2$, and $e_r$ such that $[e_i]_T=1$ in $G_T$. 

\vspace{.25cm}
\noindent\textbf{Case 2.}
Assume that $\deg(\s,\D_{[d]\setminus S})= 3$. Then $\lk (\s,\D)_S$ consists of exactly three different cycles, say $\Tilde{C}_1,\Tilde{C}_2,$ and $\Tilde{C}_3$. Since $\Gamma_{S}\geq 2$, the number of different cycles in $\D_S$ is greater than 4. Therefore, there exists a $(d-3)$-simplex $\eta$ in $\D_{[d]\setminus S}$ such that $\lk (\eta,\D)\setminus\lk (\s,\D)\neq\emptyset$.

\vspace{.25cm}
\noindent\textbf{Case 2a.} Let $d\geq 4$. By Lemma \ref{strong r-connected}, there exists a sequence of $(d-3)$-simplices $\s_1,\dots, \s_m$ such that $\s_1=\s$, $\s_m=\eta$ and $\s_i\cap\s_{i+1}$ is a simplex of dimension $(d-4)$. Considering the order of occurrences as indexed, let $\s_j$ be the first $(d-3)$ simplex in the sequence $\s_1,\dots, \s_m$ such that $\lk (\s_i,\D)\setminus\lk (\s,\D) = \emptyset$ for $1\leq i\leq j-1$ and $\lk (\s_j,\D)\setminus\lk (\s,\D)\neq\emptyset$. Without loss of generality, let $\Tilde{C}_2\subseteq \lk (\s_{j-1},\D)$, and consider the cycle $\Tilde{C}'_3\subseteq\lk (\s_j,\D)$ such that $(\lk (\s_j,\D)\setminus\lk (\s,\D))\cap \Tilde{C}'_3\neq\emptyset$ and $\Tilde{C}_2\cap \Tilde{C}'_3\neq \emptyset$. Then, there exists a vertex $u$ in $\s_j\cap\s_{j-1}$ such that  both $\Tilde{C}_2$ and $\Tilde{C}'_3$ are subcomplexes of $\lk (u,\D)$.

\vspace{.25cm}
\noindent\textbf{Case 2b.} Let $d=3$. If $a$ is a vertex in $\D_{[3]\setminus S}$ with degree three, then $\lk (a,\D)$ contains three different cycles in the graph $\D_{S}$. If all degree-three vertices in $\D_{S}$ share the same set of such cycles, then this leads to a contradiction, as one can find an edge in  $\D_{S}$ with an empty link. Note that if $P(a_0,\dots,a_n)$ is a path in $\D_{[3]\setminus S}$ such that $\deg(a_i,\D_{[3]\setminus S})=2$ for each $1\leq i\leq n-1$, then $\cap_{i=0}^{n} \lk (a_i,\D)$ contains a cycle from $\D_{S}$. On the other hand, if a pair of degree-three vertices in $\D_{[3]\setminus S}$ share at least two different cycles from $\D_S$ then they share all three in their corresponding link.  Since every pair of degree-three vertices in $\D_{[3]\setminus S}$ is connected by a path, there exists a pair of degree-three vertices, say $v$ and $u$,  in  $\D_{[3]\setminus S}$ such that $\lk (v,\D)$ and $\lk (u,\D)$ share exactly one common cycle from $\D_{S}$ (since if they share two, they must share all three). Let $\Tilde{C}_i\subseteq\lk (v,\D)$ and $\Tilde{C}'_i\subseteq\lk (u,\D)$ for $1\leq i\leq 3$, where $\Tilde{C}_2=\Tilde{C}'_2$.  

\vspace{.2cm}

Now we arrive at the final conclusion for both Cases 2a and 2b. Recall that  $\Tilde{C}_1,\Tilde{C}_2,$ and $\Tilde{C}_3$ are cycles in $\lk (\s,\D)$, and in both cases we found a vertex $u$ and a cycle $\Tilde{C}'_3$ such that $\Tilde{C}_2$ and $\Tilde{C}'_3$ are subcomplexes of $\lk (u,\D)$.

Let $e_1\in E(\Tilde{C}_1)\setminus E(\Tilde{C}_2)$ and $e_2\in E(\Tilde{C}_2)\setminus E(\Tilde{C}_1)$ be two edges. Consider the paths $T_1=\Tilde{C}_1\setminus\{e_1\}$ and $T_2=\Tilde{C}_2\setminus\{e_2\}$. Let $T^2=T_1\cup T_2$. Then, $T^2$ is a spanning tree of $\Tilde{C}_1\cup \Tilde{C}_2$. Similarly, let $e_3\in E(\Tilde{C}'_3)\setminus E(\Tilde{C}_1\cup \Tilde{C}_2)$ and consider the path $T_3=\Tilde{C}'_3\setminus\{e_3\}$. Let $T^3=T_2\cup T_3$. Then, $T^{3}$ is a spanning tree of $\Tilde{C}'_3\cup \Tilde{C}_2$. Define $\Tilde{T}':=T^2\cup T^3$. We will prove that $\Tilde{T}'$ is a spanning tree of $\Tilde{C}':=\Tilde{C}_1\cup \Tilde{C}_2\cup \Tilde{C}_3\cup \Tilde{C}'_3$. It is evident that $V(T')=V(\Tilde{C}')$. We now need to prove that $\Tilde{T}'$ is acyclic. 

If possible, suppose that $\Tilde{T}'$ contains a cycle, say $C$. Since both $T^2$ and $T^3$ are trees, $E(C)\cap E(T^2)\neq\emptyset$ and $E(C) \cap E(T^3)\neq \emptyset$. Moreover, the cycle $C$ contains edges from both $\Tilde{C}_1$ and $\Tilde{C}'_3$, and these edges from different cycles meet at some vertices of $\Tilde{C}_2$. Without loss of generality, let $u'$ and $v'$ be two such vertices in $\Tilde{C}_1\cap \Tilde{C}_2$ such that the path $\Tilde{C}_1[\{u',\dots v'\}]\subseteq C\cap T^2$ and $\Tilde{C}_1[\{u',\dots v'\}]\cap \Tilde{C}_2 =\{u',v'\}$. Therefore, we obtain a cycle $C'$ defined as $C':=(\Tilde{C}_2[\{u',\dots, v'\}]\cap T^2)\cup \Tilde{C}_1[\{u',\dots v'\}]$, which is embedded in $T^2$, contradicting the fact that $T^2$ is a tree. Hence, $\Tilde{T}'$ is a tree. 

Now, extend $\Tilde{T}'$ to a spanning tree $T$ of $G(\D)$, and consider the group $G_T$. Note that $T^2\cup\{e_i\}\subseteq \lk (\s,\D)$ for $i=1,2,$ and $T^3\cup\{e_3\}\subseteq \lk (u,\D)$. Therefore, by Lemma \ref{unique cycle Ce is in link of a vertex}, we have $[e_i]_T=1$ in $G_T$ for $1\leq i\leq 3$. 
Thus, there exists a spanning tree $T$ of $G(\D)$ and three edges $e_i$ in $E(\D_S)\setminus E(T)$ such that $[e_i]_T=1$ in $G_T$ for $1\leq i\leq 3$. This completes the proof.
\end{proof}

In the proof of Theorem \ref{d-dim balanced NPM with SS geq 2}, we observed that whenever $\Gamma_S \geq 2$, there exists a spanning tree $T$ such that a generating set of $G_T$ contains at most $\Gamma_S- 2$ elements. Therefore, if $\Gamma_S = 2$, then $G_T$ has no non-zero generators. Consequently, we arrive at the following conclusion:

\begin{Corollary}\label{d-dim balanced npm SS=2}
Let $\D$ be a balanced normal $d$-pseudomanifold, and let $S$ be a two-element subset of $[d]$ such that $\Gamma_{S} = 2$. Then $m(\D)=0$. 
\end{Corollary}
 
 
 \begin{Theorem}\label{3 dim and balanced genus 3}
Let $M$ be a $3$-manifold. Then $M$ is homeomorphic to $\mathbb{S}^3$ if and only if $\mathcal{G}_{M}\leq 3$.
\end{Theorem}
\begin{proof}
Let $\mathcal{G}_{M}\leq 3$, and let $\D$ be a balanced triangulation of $M$ satisfying $\mathcal{G}(\D)\leq 3$. From Equation \eqref{3dim rho}, for every cyclic permutation $\e=(\e_0,\e_1,\e_2,\e_3)$ of $[3]$, the balanced $\e$-genus is given by $\rho_{\varepsilon}(\D)= 1+ f_{\{\e_0,\e_2\}}-f_{\{\e_0\}}-f_{\{\e_2\}}= \Gamma_{S}+1$, where $S=\{\e_0,\e_2\}$. Since $\mathcal{G}(\D)\leq 3$, there exists a cyclic permutation $\e$ of $[3]$ for which $\rho_{\varepsilon}(\D)\leq 3$ holds, i.e., there is a 2-element subset $S$ of $[3]$ such that $\Gamma_{S} \leq 2$. If  $\Gamma_{S} \leq 1$, then Lemmas \ref{d-dim manifold with SS=0} and \ref{d-dim manifold with SS=1} imply that $\D$ is a triangulation of a 3-sphere. On the other hand, if $\Gamma_{S} = 2$, then from Corollary \ref{d-dim balanced npm SS=2} it follows that $m(\D)=0$. Therefore, $\D$ is a triangulation of $\mathbb{S}^3$.

For the converse part, recall that in Lemma \ref{minimum balanced genus of manifold} we established that for $d\geq 3$, $\mathcal{G}_{\mathbb{S}^d}=1+(d-3)2^{d-2}$. In particular, for $d=3$, we have $\mathcal{G}_{\mathbb{S}^3}=1<3$. This completes the proof.
%
\end{proof}

\vspace{.2cm}

\noindent \textit{Proof of Theorem \ref{main1}:}  Let $\D$ be a balanced triangulation of a $3$-manifold $M$. Since $\|\D\|\ncong \mathbb{S}^3$, it follows from Theorem \ref{3 dim and balanced genus 3} that  $\mathcal{G}(\D)\geq 4$. Therefore, for every cyclic permutation $\e$ of $[3]$, we have $\rho_{\varepsilon}(\D)\geq 4$, implying that $\Gamma_S\geq 3$ for every two-element subset $S$ of $[3]$. Hence, by Theorem \ref{d-dim balanced NPM with SS geq 2}, we obtain $\Gamma_S\geq m+2$ for every two-element subset $S$ of $[3]$. Thus, from Equation \eqref{3dim rho}, it follows that $\rho_{\e}(\D)\geq m+3$ for every cyclic permutation $\e$ of $[3]$. Consequently, $\mathcal{G}(\D)\geq m+3$ for every balanced triangulation $\D$ of $M$, which leads to the conclusion that  $\mathcal{G}_M\geq m+3$.\hfill$\Box$




\vspace{.25cm}

\begin{Theorem}\label{4 dim and balanced genus 2X+10}
Let $M$ be a PL $4$-manifold. Then $M$ is homeomorphic to $\mathbb{S}^4$ if and only if $\mathcal{G}_{M}\leq  2 \chi(M)+10$.
\end{Theorem}
\begin{proof}
Let $\D$ be a balanced triangulation of $M$ such that $ \mathcal{G}(\D)\leq 2 \chi(\D)+10$. Then, there exists a cyclic permutation $\e$ of $[4]$ such that $\rho_{\e}(\D)\leq 2 \chi(\D)+10$. Moreover, from Equation \eqref{4dim rho2} it follows that $f_1(\D)- \sum_{i\in\mathbb{Z}_{5}}f_{\{\e_i,\e_{i+1}\}}-2f_0(\D)\leq 9$. 

The term $f_1(\D)- \sum_{i\in\mathbb{Z}_{5}}f_{\{\e_i,\e_{i+1}\}}-2f_0(\D)$ consists of five expressions of the form  $f_{\{\e_i,\e_{i+2}\}}-f_{\{\e_i\}}-f_{\{\e_{i+2}\}}$, where $i\in\mathbb{Z}_5$. Therefore, there exists at least one pair $\{i,j\}$ such that $f_{\{\e_i,\e_j\}}-f_{\{\e_i\}}-f_{\{\e_j\}}\leq 1$. Consequently, there exists a two-element subset $S$ of $[4]$ such that $\Gamma_S\leq 1$. Therefore, by Lemmas \ref{d-dim manifold with SS=0} and \ref{d-dim manifold with SS=1}, we conclude that $\|\D\|\cong\mathbb{S}^4$.

For the converse part, recall that in Lemma \ref{minimum balanced genus of manifold} we established that for $d\geq 3$, $\mathcal{G}_{\mathbb{S}^d}=1+(d-3)2^{d-2}$. In particular, for $d=4$, we have $\mathcal{G}_{\mathbb{S}^4}=5<2\chi(\mathbb{S}^4)+10=14$. This completes the proof.
\end{proof}


\noindent \textit{Proof of Theorem \ref{main2}:} Let $\D$ be a balanced triangulation of a $4$-manifold such that $\|\D\|\ncong \mathbb{S}^4$. From Lemmas \ref{d-dim manifold with SS=0} and \ref{d-dim manifold with SS=1} it follows that $\Gamma_S\geq 2$ for every two-element subset of $[4]$. Therefore, by Theorem \ref{d-dim balanced NPM with SS geq 2}, we have $\Gamma_S\geq m+2$ for every two-element subset of $[4]$.

Note that in the expression of $\rho_{\e}(\D)$ for balanced $4$-manifolds in Equation \eqref{4dim rho2}, the term $f_1(\D)- \sum_{i\in\mathbb{Z}_{5}}f_{\{\e_i,\e_{i+1}\}}-2f_0(\D)$ is the sum of the expressions of the form $f_{\{\e_i,\e_{i+2}\}}-f_{\{\e_i\}}-f_{\{\e_{i+2}\}}$, where $i\in\mathbb{Z}_5$. Therefore, $\Gamma_S\geq m+2$ for every two-element subset of $[4]$ implies that $f_1(\D)- \sum_{i\in\mathbb{Z}_{5}}f_{\{\e_i,\e_{i+1}\}}-2f_0(\D)\geq 5(m+2)= 5m+10$. Hence,  for every cyclic permutation $\e$ of $[4]$, we obtain $\rho_{\e}(\D)\geq 2 \chi(\D)+ 5m +11$. Thus, for every balanced triangulation $\D$ of $M$, we have $ \mathcal{G}(\D)\geq 2 \chi(\D)+ 5m +11$. This completes the proof. \hfill$\Box$


\vspace{.25cm}


\begin{Remark}\label{Sharpness}
{\rm In \cite{KleeNovik2016, L.Venturello}, minimal balanced triangulations of the $(d-1)$-dimensional sphere bundle over $\mathbb{S}^1$ are provided, and the structure can be considered the balanced analog of K\"{u}hnel's constructions described in \cite{Kuhnel1986}. In particular, the $f$-vector of the minimal triangulations of the trivial bundle $\mathbb{S}^2\times \mathbb{S}^1$ and the twisted bundle $\mathbb{S}^2\rtimes \mathbb{S}^1$ are given by $(1,14,64,100,50)$ and $(1,12,54,84,42)$, respectively. From these triangulations, the calculated values of the balanced genus of the 3-dimensional sphere bundle over a circle become 4. On the other hand, the $f$-vector of the minimal triangulations of $\mathbb{S}^3\times \mathbb{S}^1$ is  given by $(1,15,90,210,225,90)$ and the calculated values of the balanced genus of $\mathbb{S}^3\times \mathbb{S}^1$ is 16. 

Note that $m(\mathbb{S}^{d-1}\times \mathbb{S}^1)=1$ and $m(\mathbb{S}^{d-1}\rtimes \mathbb{S}^1)=1$. Therefore, from Theorem \ref{main1} it follows that both $\mathcal{G}(\mathbb{S}^2\rtimes \mathbb{S}^1)$ and $\mathcal{G}(\mathbb{S}^2\rtimes \mathbb{S}^1)$ are at least 4. Similarly, from Theorem \ref{main2}, we have $\mathcal{G}(\mathbb{S}^3\times \mathbb{S}^1)$ is at least 16. Hence, our results result in Theorem \ref{main1} is sharp for $\mathbb{S}^{2}\times \mathbb{S}^1$ and $\mathbb{S}^{2}\rtimes \mathbb{S}^1$, where $\mathcal{G}_{\mathbb{S}^{2}\times \mathbb{S}^1}= \mathcal{G}_{\mathbb{S}^{2}\rtimes \mathbb{S}^1}=4$, and the bound in Theorem \ref{main2} is sharp for $\mathbb{S}^{3}\times \mathbb{S}^1$, where $\mathcal{G}_{\mathbb{S}^{3}\times \mathbb{S}^1}=16$.
%
%

}
\end{Remark}


\section{Future Directions}
 
In this article, we introduced a new concept called the balanced genus and established a lower bound for the balanced genus of PL 3- and 4-manifolds. We believe that this notion will lead to an active area of research in combinatorial topology. There are several intriguing problems related to balanced genus that warrant further investigation. Below, we outline some research directions in this area:  

\medskip

\begin{Question}   
In dimension $4$, we have provided a lower bound for the balanced genus of PL $4$-manifolds. An interesting research direction is to classify all PL $4$-manifolds based on their balanced genus.  
\end{Question}

For a cyclic permutation $\e=(\e_0,\e_1,\e_2,\e_3,\e_4)$ of $[4]$, consider the class of balanced 4-manifolds $\mathcal{D}=\{\D:  f_{\{\e_i,\e_{i+2}\}}-f_{\{\e_i\}}-f_{\{\e_{i+2}\}} = m(\D)+2, \hspace{.1cm}\text{where} \hspace{.2cm} i\in\mathbb{Z}_5\}$. If a PL $4$-manifold $M$ admits a balanced triangulation $\Delta \in \mathcal{D}$, then the balanced genus of $M$ attains our lower bound. However, the study of balanced triangulations of PL $4$-manifolds remains quite limited. Consequently, it is not yet clear which PL $4$-manifolds admit a triangulation belonging to $\mathcal{D}$. This naturally leads to the following question.

\begin{Question}  
Determine the class of PL $4$-manifolds that admit a balanced triangulation lying in $\mathcal{D}$, i.e., those for which $\mathcal{G}_M = 2\chi(M) + 5m + 11 $.
\end{Question}

\begin{Question} 
We have shown that for $ d \geq 3 $, the balanced genus of $ \mathbb{S}^d $ is given by  
   \( 1 + (d-3)  2^{d-2}. \)  
   A natural extension of this result is to determine the balanced genus of other well-known PL $ d $-manifolds.  
\end{Question} 

According to Remark \ref{Sharpness}, the calculated values of the balanced genus of two non-sphere manifolds $\mathbb{S}^2\times \mathbb{S}^1$ and $\mathbb{S}^3\times \mathbb{S}^1$ are 4 and 16, respectively. For higher dimensions, we conjecture the following:

\begin{Conjecture}
    If $d\geq 3$, then $\mathcal{G}_{\mathbb{S}^{d-1}\times \mathbb{S}^1} =  4+3(d-3) 2^{d-2}$.
\end{Conjecture}

\begin{Question} 
One may investigate general bounds for the balanced genus of PL $ d $-manifolds for $ d > 4 $ and explore classification problems in higher dimensions.  
\end{Question}

\begin{Question}  
   Since the balanced genus can be defined not only for PL $ d $-manifolds but also for normal $ d $-pseudomanifolds, it would be worthwhile to extend the above problems to the setting of normal pseudomanifolds.  
\end{Question} 

These problems open new avenues for research in combinatorial topology and may lead to further structural insights into PL and normal pseudomanifolds.

\bigskip

\noindent {\bf Acknowledgement:} The authors would like to thank the anonymous referees for their valuable comments and suggestions, which have greatly improved the manuscript.

 \medskip
 \noindent {\bf Author contributions:} The authors have contributed equally.
 \medskip
 
\noindent {\bf Funding:} The first author is supported by the Mathematical Research Impact Centric Support (MATRICS) Research Grant (MTR/2022/000036) by ANRF (India). The second author is supported by the Prime Minister's Research Fellowship (PMRF/1401215), India.

 \medskip

\noindent {\bf Data Availability:} No data was used for the research described in the article.

\medskip
 
\noindent {\bf Conflict of interest:} There is no competing interest.

\medskip

\noindent {\bf Ethical approval:} Not applicable.

\end{document}